\documentclass[10pt,reqno]{amsart}
\usepackage{graphicx, amssymb, color,pdfsync}
\usepackage[all,cmtip]{xy}
\numberwithin{equation}{section}
\usepackage{mathrsfs}
\usepackage{hyperref}
\hypersetup{
    colorlinks=true,       
    linkcolor=blue,          
    citecolor=blue,        
    filecolor=blue,      
    urlcolor=blue           
}
\usepackage{tikz}
\usetikzlibrary{matrix,arrows}
\DeclareMathOperator{\Aut}{Aut}
\usepackage[switch]{lineno}
\usepackage{yfonts}

\begin{document}
\newcommand{\s}{\vspace{0.2cm}}

\newtheorem{theo}{Theorem}
\newtheorem{prop}{Proposition}
\newtheorem{coro}{Corollary}
\newtheorem{lemm}{Lemma}
\newtheorem{claim}{Claim}
\newtheorem{example}{Example}
\theoremstyle{remark}
\newtheorem{rema}{\bf Remark}
\newtheorem{defi}{\bf Definition}

\title[On Riemann surfaces of genus $g$ with $4g-4$ automorphisms]{On Riemann surfaces of genus $g$ \\with $4g-4$ automorphisms}
\date{}

\author{Sebasti\'an Reyes-Carocca}
\address{Departamento de Matem\'atica y Estad\'istica, Universidad de La Frontera, Avenida Francisco Salazar 01145, Temuco, Chile.}
\email{sebastian.reyes@ufrontera.cl}

\thanks{Partially supported by Fondecyt Grant 11180024, Redes Etapa Inicial Grant 2017-170071 and Anillo ACT1415 PIA-CONICYT Grant}

\begin{abstract} In this article we study compact Riemann surfaces with a non-large group of automorphisms of maximal order; namely, compact Riemann surfaces of genus $g$ with a group of automorphisms of order $4g-4.$ Under the assumption that $g-1$ is prime, we provide a complete classification of them and determine isogeny decompositions of the corresponding Jacobian varieties.
\end{abstract}
\maketitle

\section{Introduction and statement of the results}

The classification of finite groups of automorphisms of compact Riemann surfaces is a classical and interesting problem which has attracted a considerable interest, going back to contributions of Wiman, Klein, Schwartz and Hurwitz, among others. 

\s

It is classically known that the full automorphism group of a compact Riemann surface of genus $g \geqslant 2$ is finite, and that its order is bounded by $84g-84.$ This bound is sharp for infinitely many values of $g,$ and a Riemann surface with this maximal number of automorphisms is characterized as a branched regular cover of the projective line with three branch values, marked with 2, 3 and 7.

\s
A result due to Wiman asserts that the largest cyclic group of automorphisms of a  Riemann surface of genus $g \geqslant 2$ has order at most $4g+2.$ Moreover, the curve \begin{equation} \label{eWiman}y^2=x^{2g+1}-1\end{equation}shows that this upper bound is sharp for each genus; see \cite{Wi} and \cite{Harvey1}. 

Besides, Accola and Maclachlan independently proved that for fixed $g \geqslant 2$ the maximum among the orders of the full automorphism group of Riemann surfaces of genus $g$ is at least $8g+8.$ Moreover, the curve \begin{equation} \label{eAccola}y^2=x^{2g+2}-1\end{equation}shows that this lower bound is sharp for each genus; see \cite{Accola} and \cite{Mac}.

\s

An interesting problem is to understand the extent to which the order of the full automorphism group determines the Riemann surface. In this regard, Kulkarni proved that for $g$ sufficiently large the curve \eqref{eWiman} is the unique Riemann surface with an automorphism of order $4g+2,$ and for $g \not\equiv -1 \mbox{ mod } 4$ sufficiently large the curve \eqref{eAccola} is the unique Riemann surface with $8g+8$ automorphisms; see \cite{K1}.
\s

Let $a,b \in \mathbb{Z}.$ Following \cite{Kis}, the sequence $ag+b$ for $g=2,3, \ldots$ is called {\it admissible} if for infinitely many values of $g$ there is a compact Riemann surface of genus $g$ with a group of automorphisms of order $ag+b.$ 

In addition to the mentioned admissible sequences $84g-84, 4g+2$ and $8g+8,$ very recently the cases $4g+4$ and $4g$ have been studied; see \cite{BCI}, \cite{CI} and also \cite{yojpaa}.

\s

Let $a \geqslant 7.$ The admissible sequence $ag-a$ has been considered by Belolipetsky and Jones in \cite{BJ}. Concretely, they succeeded in proving that under the assumption that $g-1$ is a sufficiently large prime number, a compact Riemann surface of genus $g$ with a group of automorphisms of order $ag-a$ lies in one of six infinite well-described sequences of examples. The cases $a=5$ and $a=6$ have been recently classified by Izquierdo and the author in \cite{IRC}.

\s
All the aforementioned cases are examples of compact Riemann surfaces possessing a so-called {\it large} group of automorphisms: namely, a group of automorphisms of order greater than $4g-4,$ where $g$ is the genus. In this case, it is known that the Riemann surface is either quasiplatonic (it does not admit non-trivial deformations in the moduli space  with its automorphisms) or belongs to a complex one-dimensional family in such a way that the signature of the action is  $$(0; 2,2,2,n) \mbox{ for } n \geqslant3 \mbox{ or } (0; 2,2,3,n) \mbox{ for } 3 \leqslant n \leqslant 5.$$

Riemann surfaces with large groups of automorphisms have been considered from different points of view; see, for example, \cite {CK}, \cite{K2}, \cite{AB}, \cite{Mat},  \cite{si3}, \cite{Swi} and \cite{Wi}.

\s

In this article we consider compact Riemann surfaces admitting a non-large group of automorphisms of maximal order; concretely, we study and classify those compact Riemann surfaces of genus $g$ with a group of automorphisms of order $4g-4,$ under the assumption that $g-1$ is a prime number. 

\s

\begin{theo}\label{teo1} Let $g \geqslant8$ such that $g-1$ is prime, and let $S$ be a compact Riemann surface of genus $g$ with a group of automorphisms of order $4g-4$.

If $g\equiv 0 \mbox{ mod } 4$ then $S$ belongs to $\bar{\mathcal{F}}_g^2$, where $\bar{\mathcal{F}}_g^2$ is the complex two-dimensional equisymmetric family of  compact Riemann surfaces of genus $g$ with a group of automorphisms $G_2$ isomorphic to the dihedral group $$\langle r,s : r^{2(g-1)}=s^2=(sr)^2=1\rangle,$$ such that the signature of the action of $G_2$  is $(0; 2, 2, 2, 2, 2).$

If $g \equiv 2 \mbox{ mod } 4$ then $S$ belongs to either $\bar{\mathcal{F}}_g^2$ or $\bar{\mathcal{F}}_g^1,$ where $\bar{\mathcal{F}}_g^1$ is the complex one-dimensional equisymmetric family of compact Riemann surfaces of genus $g$ with a group of automorphisms $G_1$ isomorphic to  $$\langle a,b : a^{g-1}=b^4=1, bab^{-1}=a^{r} \rangle,$$where $r$ is a $4$-th primitive root of the unity in the field of $g-1$ elements,  such that the signature of the action of $G_1$  is $(0; 2, 2, 4, 4).$ 
\end{theo}

By virtue of Dirichlet's prime number theorem, the congruences of Theorem \ref{teo1} are satisfied for infinitely many values. The genera $g=3,4$ and $6$ are exceptional in the sense that, in addition to the families before introduced, appear finitely many quasiplatonic Riemann surfaces (see \cite{bar}, \cite{Bo}, \cite{Brou}, \cite{C}, \cite{CI2}, \cite{Ku} and \cite{magaard}).  

\s

As proved by Costa and Izquierdo in \cite{CI}, the largest order of the full automorphism group of a complex one-dimensional family of Riemann surfaces of genus $g$, appearing in all genera, is $4g+4.$ On the other hand, $4g-4$ is the maximal possible order of the full automorphism group of a complex two-dimensional family of compact Riemann surfaces of genus $g.$  It is worth mentioning that the existence of the family $\bar{\mathcal{F}}_g^2$ shows that this upper bound is attained in each genus, with $g-1$ not necessarily prime.

\s

Let $t=1,2.$ The family $\bar{\mathcal{F}}_g^t$ is a closed algebraic subvariety of the moduli space of compact Riemann surfaces of genus $g$; we shall denote by ${\mathcal{F}}_g^t$ its interior,  and by $$\partial(\bar{\mathcal{F}}_g^t)=\bar{\mathcal{F}}_g^t \setminus\mathcal{F}_g^t$$its boundary. Based on classical results due to Singerman \cite{singerman2} and on Belolipetsky and Jones' classification \cite{BJ}, we are  able to describe the interior and (up to possibly finitely many exceptional cases in small genera) the boundary of the families $\bar{\mathcal{F}}_g^t.$ 

More precisely: 

\begin{theo} \label{frontera}
Let $g \geqslant8$ such that $g-1$ is prime. 

For $g \equiv 2 \mbox{ mod }4,$ the interior ${\mathcal{F}}_g^1$ of the family $\bar{\mathcal{F}}_g^1$ consists of those Riemann surfaces for which $G_1$ is the full automorphism group. Moreover, there is a positive integer $\epsilon_1$ such that if $g \geqslant\epsilon_1$  then
\begin{displaymath}
\partial(\bar{\mathcal{F}}_g^1)  = \left\{ \begin{array}{ll}
  \{X_1, X_2\} & \textrm{if $g \equiv 2 \mbox{ mod } 8$}\\
 \,\,\,\,\, \,\,\,\,\, \emptyset & \textrm{if $g \not\equiv 2 \mbox{ mod } 8,$}

  \end{array} \right.
\end{displaymath}where $X_1$ and $X_2$ are the two non-isomorphic compact Riemann surfaces of genus $g$ with full automorphism group of order $8g-8$.

The interior ${\mathcal{F}}_g^2$ of the family $\bar{\mathcal{F}}_g^2$ consists of those Riemann surfaces for which $G_2$ is the full automorphism group. Moreover, there is a positive integer $\epsilon_2$ such that  if $g \geqslant\epsilon_2$ then \begin{displaymath}
\partial(\bar{\mathcal{F}}_g^2)  = \left\{ \begin{array}{ll}
 \{Y_1, Y_2\} & \textrm{if $g \equiv 2 \mbox{ mod } 3$}\\
 \,\,\,\,\, \,\,\,\,\, \emptyset & \textrm{if $g \not\equiv 2 \mbox{ mod } 3,$}

  \end{array} \right.
\end{displaymath}where $Y_1$ and $Y_2$ are the two non-isomorphic compact Riemann surfaces of genus $g$ with full automorphism group of order $12g-12$.
\end{theo}

We recall that the Jacobian variety $JS$ of a compact Riemann surface $S$ of genus $g$  is an irreducible principally polarized abelian variety of dimension $g.$ The relevance of  the Jacobian varieties  lies in the well-known Torelli's theorem, which establishes that two compact Riemann surfaces are isomorphic if and only if the corresponding Jacobian varieties  are isomorphic as principally polarized abelian varieties. See, for example,  \cite{bl} and \cite{debarre}.

If $H$ is a group of automorphisms of  $S$ then the associated regular covering map $\pi : S \to S_H$ given by the action of $H$ on $S$  induces a homomorphism $$\pi^*: JS_H \to JS$$between the corresponding Jacobians; the set $\pi^*(JS_H)$ is an abelian subvariety of $JS$ which is isogenous to $JS_H.$ We keep the same notations as in Theorem \ref{teo1}.

\begin{theo} \label{teo3}Let $g \geqslant8$ such that $g-1$ is prime.

If $S \in \bar{\mathcal{F}}_g^1$ then the Jacobian variety $JS$ decomposes, up to isogeny, as the product $$JS \sim JS_{\langle a \rangle} \times (JS_{\langle b \rangle})^4.$$

If $S \in \bar{\mathcal{F}}_g^2$ then the Jacobian variety $JS$ decomposes, up to isogeny, as the product $$JS \sim JS_{\langle r \rangle} \times JS_{\langle s \rangle} \times JS_{\langle sr \rangle}.$$
\end{theo}

\s

The article is organized as follows.

\begin{enumerate}
\item In Section \ref{preli} we shall review the basic background: namely, Fuchsian groups,  actions on compact Riemann surfaces, the equisymmetric stratification of the moduli space and the decomposition of Jacobian varieties. 
\item In Section \ref{existencia} we shall prove the existence of the families $\bar{\mathcal{F}}_g^1$ and $\bar{\mathcal{F}}_g^2.$ \item Theorems \ref{teo1}, \ref{frontera} and \ref{teo3} will be proved in Section \ref{unicidad}, \ref{fronterac} and \ref{jaco} respectively. 
\item Finally, and for the sake of completeness, in Section \ref{noprimo} two examples will be constructed to show that Theorem \ref{teo1} is false if $g-1$ is not  prime. 
\end{enumerate}
\section{Preliminaries} \label{preli}

\subsection{Fuchsian groups and group actions on Riemann surfaces} By a {\it Fuchsian group} we mean a discrete group of automorphisms of the upper-half plane $$\mathbb{H}=\{z \in \mathbb{C}: \mbox{Im}(z) >0 \}.$$  

If $\Delta$ is a Fuchsian group and the orbit space $\mathbb{H}_{\Delta}$ given by the action of $\Delta$ on $\mathbb{H}$ is  compact, then the algebraic structure of $\Delta$ is determined by its {\it signature}: \begin{equation} \label{sig} \sigma(\Delta)=(h; m_1, \ldots, m_l),\end{equation}where $h$ is the genus of the quotient surface $\mathbb{H}_{\Delta}$ and $m_1, \ldots, m_l$ are the branch indices in the universal canonical projection $\mathbb{H} \to \mathbb{H}_{\Delta}.$ If $l=0$ then $\Delta$ is called a {\it surface Fuchsian group}.

\s

Let $\Delta$ be a Fuchsian group of signature \eqref{sig}. Then 
\begin{enumerate}
\item $\Delta$ has a canonical presentation given by generators $a_1, \ldots, a_{h}$, $b_1, \ldots, b_{h},$ $ x_1, \ldots , x_l$ and relations
\begin{equation}\label{prese}x_1^{m_1}=\cdots =x_l^{m_l}=\Pi_{i=1}^{h}[a_i, b_i] \Pi_{i=1}^l x_i=1,\end{equation}where $[u,v]$ stands for the commutator $uvu^{-1}v^{-1},$
\item the elements of $\Delta$ of finite order are conjugate to powers of $x_1, \ldots, x_l.$
\item the Teichm\"{u}ller space of $\Delta$ is a complex analytic manifold homeomorphic to the complex ball of dimension $3h-3+l$, and
\item the hyperbolic area of each fundamental region of $\Delta$ is  $$\mu(\Delta)=2 \pi [2h-2 + \Sigma_{i=1}^l(1-\tfrac{1}{m_i})].$$ 
\end{enumerate}

Let $\Gamma$ be a group of automorphisms of $\mathbb{H}.$ If $\Delta$ is a subgroup of $\Gamma$ of finite index then $\Gamma$ is also Fuchsian and they are related by the Riemann-Hurwitz formula $$\mu(\Delta)= [\Gamma : \Delta] \cdot \mu(\Gamma).$$

\s

Let $S$ be a compact Riemann surface and let $\mbox{Aut}(S)$ denote its full automorphism group. It is said that a finite group $G$ acts on $S$ if there is a group monomorphism $\psi: G\to \Aut(S).$ The space of orbits $S_G$ of the action of $G$ on $S$ is endowed with a Riemann surface structure such that the  projection $S \to S_G$ is holomorphic. 

\s

Compact Riemann surfaces and group actions can be understood in terms of Fuchsian groups as follows. By uniformization theorem (see, for example, \cite[p. 203]{FK}), there is a (uniquely determined, up to conjugation) surface Fuchsian group $\Gamma$ such that $S$ and $\mathbb{H}_{\Gamma}$ are isomorphic. Moreover, $G$ acts on $S \cong \mathbb{H}_{\Gamma}$ if and only if there is a Fuchsian group $\Delta$ containing $\Gamma$ together with a group  epimorphism \begin{equation*}\label{epi}\theta: \Delta \to G \, \, \mbox{ such that }  \, \, \mbox{ker}(\theta)=\Gamma.\end{equation*}

In this case, it is said that $G$ acts on $S$ with signature $\sigma(\Delta)$ and that this action is {\it  represented} by the epimorphism $\theta.$ If $G$ is a subgroup of $G'$ then the action of $G$ on $S$ is said to {\it extend} to an action of $G'$ on $S$ if:\begin{enumerate}
\item there is a Fuchsian group $\Delta'$ containing $\Delta.$ 
\item the Teichm\"{u}ller spaces of $\Delta$ and $\Delta'$ have the same dimension, and
\item there exists an epimorphism $$\Theta: \Delta' \to G' \, \, \mbox{ in such a way that }  \, \, \Theta|_{\Delta}=\theta.$$
\end{enumerate} 

An action is termed {\it maximal} if it cannot be extended. A complete list of signatures of Fuchsian groups $\Delta$ and $\Delta'$ for which it might be possible to have an extension as before was determined by Singerman in \cite{singerman2}. 

\subsection{Actions and equisymmetric stratification} \label{strati} Let $\text{Hom}^+(S)$ denote the group of orientation preserving homeomorphisms of $S.$ Two actions $\psi_i: G \to \mbox{Aut}(S)$  are  {\it topologically equivalent} if there exist $\omega \in \Aut(G)$ and $f \in \text{Hom}^+(S)$ such that
\begin{equation}\label{equivalentactions}
\psi_2(g) = f \psi_1(\omega(g)) f^{-1} \hspace{0.5 cm} \mbox{for all} \,\, g\in G.
\end{equation}

Each homeomorphism $f$ satisfying \eqref{equivalentactions} yields an automorphism $f^*$ of $\Delta$ where $\mathbb{H}_{\Delta} \cong S_G$. If $\mathscr{B}$ is the subgroup of $\mbox{Aut}(\Delta)$ consisting of them, then $\mbox{Aut}(G) \times \mathscr{B}$ acts on the set of epimorphisms defining actions of $G$ on $S$ with signature $\sigma(\Delta)$ by $$((\omega, f^*), \theta) \mapsto \omega \circ \theta \circ (f^*)^{-1}.$$  

Two epimorphisms $\theta_1, \theta_2 : \Delta \to G$ define topologically equivalent actions if and only if they belong to the same $(\mbox{Aut}(G) \times \mathscr{B})$-orbit (see \cite{Brou}; also \cite{Harvey} and \cite{McB}). We remark that if the genus $h=g_{S_G}$ of $S_G$ is zero and  $$\Delta=\langle x_1, \ldots, x_l : x_1^{m_1}=\cdots = x_l^{m_l}=x_1 \cdots x_l=1\rangle,$$then $\mathscr{B}$ is generated by the {\it braid transformations} $\Phi_{i} \in \mbox{Aut}(\Delta)$ defined by $$x_i \mapsto x_{i+1}, \hspace{0.3 cm}x_{i+1} \mapsto x_{i+1}x_{i}x_{i+1}^{-1} \hspace{0.3 cm} \mbox{ and }\hspace{0.3 cm} x_j \mapsto x_j \mbox{ when }j \neq i, i+1$$for each $i \in \{1, \ldots, l-1\}.$ See, for example, \cite[p. 31]{trenzas}.

\s

Let $\mathscr{M}_g$ denote the moduli space of compact Riemann surfaces of genus $g \geqslant2.$ It is well-known  that $\mathscr{M}_g$ is endowed with a structure of complex analytic space of dimension $3g-3,$ and that for $g \geqslant4$ its singular locus $\mbox{Sing}(\mathscr{M}_g)$ agrees with the set of points representing compact Riemann surfaces with non-trivial automorphisms. 

\s


Following \cite{b}, the singular locus $\mbox{Sing}(\mathscr{M}_g)$ admits an {\it equisymmetric stratification} $\{ \mathscr{M}_g^{G, \theta}\}$, where 
each {\it equisymmetric stratum} $\mathscr{M}_g^{G, \theta}$, if nonempty, corresponds to one topological class of maximal actions. More precisely:
\begin{enumerate}
\item the closure  $\bar{\mathscr{M}}_g^{G, \theta}$ of $\mathscr{M}_g^{G, \theta}$ consists of those Riemann surfaces of genus $g$ admitting an 
action of the group $G$ with fixed topological class given by $\theta,$
\item $\bar{\mathscr{M}}_g^{G, \theta}$ is a closed  irreducible algebraic subvariety of $\mathscr{M}_g,$  
 
\item if the {\it stratum} ${\mathscr{M}}_g^{G, \theta}$ is nonempty, then it is a smooth, connected, locally closed algebraic subvariety of $\mathscr{M}_g$ which is Zariski dense in $\bar{\mathscr{M}}_g^{G, \theta},$
\item there are finitely many distinct strata, and $$\mbox{Sing}(\mathscr{M}_g) = \cup_{G \neq 1, \theta} \bar{\mathscr{M}}_g^{G, \theta}.$$
\end{enumerate}

Let $\bar{\mathcal{F}}$ be a (closed) family of compact Riemann surfaces of genus $g$ such that  each of its members has a group of automorphisms isomorphic to $G.$  The family $\bar{\mathcal{F}}$ is termed {\it equisymmetric} if its interior $\mathcal{F}$ consists of only one stratum.

\subsection{Decomposition of Jacobians} \label{jacos}

If a finite group $G$ acts on a compact Riemann surface $S$ then it is known that this action induces an isogeny decomposition \begin{equation} \label{iso}JS \sim A_1 \times \ldots \times A_{r}\end{equation} which is $G$-equivariant; see \cite{cr} and \cite{l-r}. The factors $A_j$ in \eqref{iso} are in bijective correspondence with the rational irreducible representations of $G$. If the factor $A_1$ is associated with the trivial representation of $G,$ then $A_1 \sim JS_G.$

\s

The decomposition of Jacobians with group actions has been extensively studied, going back to  Wirtinger, Schottky and Jung (see, for example,  \cite{SJ} and \cite{W}). For decompositions of Jacobians with respect to special groups, we refer to the articles \cite{d1}, \cite{nos}, \cite{IJR}, \cite{PA}, \cite{d3} and \cite{RCR}.
%
%
%
%
\s

Assume that $G$ acts on a compact Riemann surface $S$ with signature  \eqref{sig},  and that this action is determined by the epimorphism $\theta : \Delta \to G,$ where $\Delta$ is written with its canonical presentation \eqref{prese}. We define $\mathfrak{J}_{\theta}$ as the set of complex irreducible representations $V$ of $G$ characterized as follows:\begin{enumerate}
\item the trivial representation belongs to $\mathfrak{J}_{\theta}$ if and only if $h\neq 0,$  and  
\item a non-trivial  representation $V$ belongs to $\mathfrak{J}_{\theta}$ if and only if \begin{equation*}\label{dimensiones}
d_{V}(h -1)+\tfrac{1}{2}\Sigma_{i=1}^l (d_{V}-d_{V}^{\langle \theta(x_i) \rangle} ) \neq 0,\end{equation*}where $d_V$ is the degree of $V$ and  $d_{V}^{\langle \theta(x_i) \rangle}$  is the dimension of the subspace of $V$ fixed under the action of the subgroup of $G$ generated by $\theta(x_i).$

\end{enumerate}

Let $H_1, \ldots, H_t$ be groups of automorphisms  of $S$ such that $G$ contains $H_i$ for each $i$. Following \cite{kanirubiyo} (and using \cite[Theorem 5.12]{yoibero}), the collection $\{H_1, \ldots, H_t\}$ is termed {\it $G$-admissible} if 
$$d_{V}^{H_1}+ \cdots + d_{V}^{H_t}  \leqslant d_{V} \, \mbox{ for each } \, V \in \mathfrak{J}_{\theta},$$ and is called {\it admissible} if it is $G$-admissible for some group $G$. If $\{H_1, \ldots, H_t\}$ is admissible then, by \cite{kanirubiyo}, $JS$ decomposes, up to isogeny, as
$$JS \sim \Pi_{i=1}^t JS_{H_i} \times P$$for some abelian subvariety $P$ of $JS.$ See also \cite{KR}.

\subsection*{\it Notation} Let $n \geqslant2$ be an integer. 
Throughout this article we shall denote by $C_n$ the cyclic group of order $n$ and by $\mathbf{D}_n$ the dihedral group of order $2n.$


\section{Existence of the families $\bar{\mathcal{F}}_g^1$ and $\bar{\mathcal{F}}_g^2$}\label{existencia}

\begin{prop} \label{oja} Let $g \geqslant6$ such that $g-1$ is a prime number and $g \equiv 2 \mbox{ mod } 4.$  There exists a complex one-dimensional equisymmetric family $\bar{\mathcal{F}}_g^1$ of compact Riemann surfaces of genus $g$ with a group of automorphisms isomorphic to \begin{equation*}  \langle a,b : a^{g-1}=b^4=1, bab^{-1}=a^{r} \rangle=C_{g-1} \rtimes_4 C_4,\end{equation*}where $r$ is a $4$-th primitive root of the unity in the field of $g-1$ elements, such that the signature of the action  is $(0; 2,2,4,4).$ 
\end{prop} 

\begin{proof} Set $q=g-1,$ and let $\Delta$ be a Fuchsian group of signature $\sigma=(0; 2,2,4,4)$ with canonical presentation $$\Delta=\langle x_1, x_2, x_3, x_4 : x_1^2=x_2^2=x_3^4=x_4^4=x_1x_2x_3x_4=1 \rangle.$$ The epimorphism $\Theta: \Delta \to C_{q} \rtimes_4 C_4$ defined by $$\Theta(x_1)=b^2, \,\, \Theta(x_2)=ab^2, \,\, \Theta(x_3)=ab, \,\,  \Theta(x_4)=b^3$$guarantees the existence of a complex one-dimensional family $\bar{\mathcal{F}}_g^1$ of compact Riemann surfaces $S$ of genus $g$ with a group of automorphisms $G$ isomorphic to $C_{q} \rtimes_4 C_4$ acting on $S$ with signature $\sigma.$ 

To prove that $\bar{\mathcal{F}}_g^1$ consists of only one stratum, we firstly notice that  the involutions of $C_{q} \rtimes_4 C_4$ are $a^lb^2$ and the elements of order 4 are $a^lb$ and $a^lb^3$ for $1 \leqslant l \leqslant q.$ Then,  up to a permutation, an epimorphism $\theta: \Delta \to C_{q} \rtimes_4 C_4$ representing an action of $G$ on $S$ is defined by$$\theta(x_1)=a^{l_1}b^2, \,\, \theta(x_2)=a^{l_2}b^2, \,\, \theta(x_3)=a^{l_3}b,  \,\, \theta(x_4)=a^{l_4}b^3,$$for some $1 \leqslant l_1, \ldots, l_4 \leqslant  q.$ After applying an inner automorphism of $G,$ we can assume that $l_4 \equiv 0 \mbox{ mod }q$ and therefore $l_2 \equiv l_1 + l_3 \mbox{ mod } q.$ Note that if $l_1 \equiv l_3 \equiv 0 \mbox{ mod } q$ then $\theta$ is not surjective; thus,  without loss of generality, we can assume $l_ 3 \not\equiv 0 \mbox{ mod } q.$ Now, consider the automorphism of $G$ given by $a \mapsto a^{t_3}, b \to b,$ where $l_3t_3 \equiv 1 \mbox{ mod } q,$ to obtain that $\theta$ is equivalent to the epimorphism $\theta_n$ defined by $$\theta_n(x_1)=a^{n}b^2, \,\, \theta_n(x_2)=a^{n+1}b^2, \,\, \theta_n(x_3)=ab, \,\, \theta_n(x_4)=b^3$$for some $1 \leqslant n \leqslant q.$ Finally, as $\Phi_{1} \cdot \theta_n=\theta_{n+1},$ each $\theta_n$  is equivalent to $\theta_0=\Theta.$   
\end{proof}

\begin{prop} \label{t4q2} There exists a complex two-dimensional family $\bar{\mathcal{F}}_g^2$ of compact Riemann surfaces of genus $g \geqslant2$ with a group of  automorphisms isomorphic to the dihedral group of order $4g-4$ such that the signature of the action is $(0; 2,2,2,2,2).$ If, in addition, $g-1$ is prime then the family is equisymmetric. 
\end{prop}

\begin{proof} Set $q=g-1,$ and let $\Delta$ be a Fuchsian group of signature $\sigma=(0; 2,2,2,2,2)$ with canonical presentation $$\Delta=\langle x_1, x_2, x_3, x_4, x_5 : x_1^2=x_2^2=x_3^2=x_4^2=x_5^2=x_1x_2x_3x_4x_5=1 \rangle.$$ The epimorphism $\Theta: \Delta \to \mathbf{D}_{2q}=\langle r, s : r^{2q}=s^2=(sr)^2=1\rangle$ defined by $$\Theta(x_1)=s, \,\, \Theta(x_2)=s, \,\, \Theta(x_3)=sr^{q+1}, \,\, \Theta(x_4)= sr, \,\, \Theta(x_5)=r^{q}$$guarantees the existence of a complex two-dimensional family $\bar{\mathcal{F}}_g^2$ of compact Riemann surfaces $S$ of genus $g$ with a group of automorphisms $G$ isomorphic to $\mathbf{D}_{2q}$ acting on $S$ with signature $\sigma.$ 

We now assume $q$ to be prime and proceed to prove that $\bar{\mathcal{F}}_g^2$ is equisymmetric.  Let $\theta: \Delta \to \mathbf{D}_{2q}$ be an epimorphism representing an action of $G$ on $S.$ Note that the involutions of $G$ are $r^q$ and $sr^t$ for $1 \leqslant t \leqslant2q,$ and if some $\theta(x_i)$ equals $r^q$ then, after considering suitable braid automorphisms, it can be supposed $i=5.$

\s

We claim that one and only one among the elements $\theta(x_i)$ equals $r^q.$ Indeed, if $l$ denotes the number of elements $\theta(x_i)$ which are equal to $r^q,$ then clearly $l$ is different from 4 and 5  because otherwise $\theta$ is not surjective. If $l=3$ then it can be supposed $\theta(x_1x_2)=r^{q}.$ If we write $\theta(x_i)=sr^{t_i}$ then $t_2-t_1 \equiv q \mbox{ mod } 2q,$ showing that $t_2-t_1$ is odd and therefore, after considering the automorphism of $G$ given by $r \mapsto r, s \mapsto sr,$ we can assume that $t_1$ is even and that $t_2$ is odd. Now, we apply an appropriate inner automorphism of $G$ to suppose that $t_1\equiv 0 \mbox{ mod } 2q$ and  $t_2\equiv q \mbox{ mod }2q.$ The contradiction is obtained by noticing that $\langle s, r^q\rangle \cong C_2^2.$ Similarly, if $l=2$ then we can suppose $\theta(x_1x_2x_3)=1.$ If we write $\theta(x_i)=sr^{t_i}$ then $sr^{t_1-t_2+t_3}=1,$ which is not possible.

It follows that, up to equivalence, the epimorphism $\theta : \Delta \to \mathbf{D}_{2q}$ is given by $$\theta(x_1)=sr^{t_1}, \, \theta(x_2)=sr^{t_2}, \, \theta(x_3)=sr^{t_3}, \, \theta(x_4)= sr^{t_4}, \, \theta(x_5)=r^{q}$$for some $1 \leqslant t_1, \ldots, t_4 \leqslant 2q$ which satisfy $t_2-t_1+t_4-t_3 \equiv  q \mbox{ mod } 2q.$ Now, after considering, if necessary, braid automorphisms and the automorphism of $G$ given by $r \mapsto r, s \mapsto sr,$ we can suppose $t_1, t_2, t_3$ to be even and $t_4$ to be odd. Furthermore, after applying a suitable inner automorphism of $G$, we can assume that $t_1=0.$

If $t_4 \neq q$ then we apply the automorphism of $G$ given by $r \mapsto r^{l_4}, s \mapsto s,$ where $t_4l_4 \equiv 1 \mbox{ mod } 2q,$ to see that $\theta$ is equivalent to the epimorphism $\theta_n$ given by $$\theta_n(x_1)=s, \, \theta_n(x_2)=sr^{n}, \, \theta_n(x_3)=sr^{n+q+1}, \, \theta_n(x_4)= sr, \, \theta_n(x_5)=r^{q}$$for some $1 \leqslant n\leqslant 2q$ even. The equality $\Phi_{2}^2 \cdot \theta_{n} = \theta_{n-2}$ shows that $\theta_n$ is equivalent to $\theta_0=\Theta.$ Similarly, if now  $t_4 = q$ then  $t_2 \equiv t_3 \mbox{ mod } 2q;$ we write $t=t_2=t_3.$ We apply the inner automorphism of $G$ induced by $sr^{t/2}$ and then the braid automorphism $\Phi_{2} \circ \Phi_{1}$ to see that $\theta$ is equivalent to the epimorphism $\theta_t$ defined by $$\theta_t(x_1)=s, \, \theta_t(x_2)=s, \, \theta_t(x_3)=sr^{t}, \, \theta_t(x_4)= sr^{t+q}, \, \theta_t(x_5)=r^{q}$$where $t \not\equiv 0 \mbox{ mod } 2q.$ Finally, consider the automorphism of $G$ given by $r \mapsto r^l, s \mapsto s,$ where $(q+t)l \equiv  1 \mbox{ mod } 2q,$ to see that $\theta_t$ is equivalent to $\Theta.$
\end{proof}


\section{Proof of Theorem \ref{teo1}} \label{unicidad}

Set $q=g-1$ and let $S$ be a compact Riemann surface of genus $g \geqslant8$ with a group of automorphisms $G$ of order $4q$ where $q$ is prime. By the Riemann-Hurwitz formula the possible signatures of the action of $G$ on $S$ are  $$(1; 2), \, (0; 2,2,4,4) \, \mbox{ and } \, (0; 2,2,2,2,2).$$

By the classical  Sylow's theorems if $q \equiv 3 \mbox{ mod } 4$ then $G$ is isomorphic to either $${C}_{4q}, \,\, {C}_q \times {C}_2^2,\,\,  \mathbf{D}_{2q} \,\, \mbox{ or } \,\, \langle a,b : a^q=b^4=1, bab^{-1}=a^{-1}\rangle=C_q \rtimes_2 C_4,$$and if $q \equiv 1 \mbox{ mod } 4$ then, in addition to these groups, $G$ can be isomorphic to $$\langle a,b : a^q=b^4=1, bab^{-1}=a^{r} \rangle=C_q \rtimes_4 C_4$$where $r$ is a $4$-th primitive root of the unity in the field of $q$ elements. 

\s

The proof of Theorem \ref{teo1} is a consequence of the following three claims.

\s
 
{\bf Claim 1.} The following statements are equivalent.
\begin{enumerate}
\item $S$ is a compact Riemann surface of genus $g$ with a group of automorphisms of order $4q$ acting on $S$ with signature $\sigma=(0; 2,2,4,4).$
\item $g \equiv 2 \mbox{ mod } 4$ and  $S \in \bar{\mathcal{F}}_g^1.$ 
\end{enumerate}
\s

Let us assume that $S$ is a compact Riemann surface of genus $g$ with a group of automorphisms $G$ of order $4q$ acting on $S$ with signature $\sigma.$ Let $\Delta$ be a Fuchsian group of signature $\sigma$ with canonical presentation $$\Delta= \langle x_1, x_2, x_3, x_4 : x_1^2=x_2^2=x_3^4=x_4^4=x_1x_2x_3x_4=1 \rangle,$$ and assume the action of $G$ on $S$ to be represented by the epimorphism $\theta: \Delta \to G.$

First of all, note that $G$ cannot be isomorphic to $C_q \times C_2^2$ or $\mathbf{D}_{2q}$ because they do not have elements of order 4, and cannot be isomorphic to $C_{4q}$ because otherwise  $C_{4q}$ would be generated by elements of order 2 and 4. We claim that $G$ cannot be isomorphic to $C_q \rtimes_2 C_4$ either. Indeed, as $b^2$ is the unique involution of $C_q \rtimes_2 C_4$ and as its elements of order 4 are $a^tb$ and $a^tb^3$ for $1 \leqslant t \leqslant q,$ after considering the automorphism of $C_q \rtimes_2 C_4$ given by $a \mapsto a, b \mapsto b^{3},$ the epimorphism $\theta$ could only be defined either by
\begin{enumerate}
\item $\theta(x_1)=b^2, \, \theta(x_2)=b^2, \, \theta(x_3)=a^{t_1}b, \, \theta(x_4)=a^{t_2}b,$ or 
\item $\theta(x_1)=b^2, \, \theta(x_2)=b^2, \, \theta(x_3)=a^{t_1}b, \, \theta(x_4)=a^{t_2}b^3,$
\end{enumerate}for some $1 \leqslant t_1, t_2 \leqslant q$. The former case cannot yield an action because the image of $x_1x_2x_3x_4$ is different from $1$ for each possible choice of $t_1$ and $t_2$. Similarly, the latter case could give rise to an action only if $t_1 \equiv t_2 \mbox{ mod } q;$ however, in this case $\theta$ would not be surjective. 

All of the above shows that $g \equiv 2 \mbox{ mod } 4$ and that $G$ is isomorphic to $C_q \rtimes_4 C_4;$ consequently $S \in \bar{\mathcal{F}}_g^1$. The converse is direct, and the proof of the claim is done.

\s

{\bf Claim 2.} There is no a compact Riemann surface $S$ of genus $g$ with a group of automorphisms of order $4q$ acting on it with signature $(1;2).$

\s

Let $G$ be a group of order $4q$ and let $\Delta$ be a Fuchsian group of signature $(1;2).$ For every torsion-free kernel epimorphism $$\Delta= \langle a_1, b_1, x_1 : [a_1, b_1] x_1=x_1^2=1 \rangle \to G,$$ the image of $x_1$ must belong to the commutator subgroup of $G.$ Thus, to conclude suffice to notice that the commutator subgroup of each group of order $4q$ does not contain involutions.

\s

{\bf Claim 3.}   The following statements are equivalent.
\begin{enumerate}
\item $S$ is a compact Riemann surface of genus $g$ with a group of automorphisms of order $4q$ acting on $S$ with signature $\sigma=(0;2,2,2,2,2).$
\item  $S \in \bar{\mathcal{F}}_g^2.$ 
\end{enumerate}
\s

Let us assume that $S$ is a compact Riemann surface of genus $g$ with a group of automorphisms $G$ of order $4q$ acting on $S$ with signature $\sigma.$ Let $\Delta$ be a Fuchsian group of signature $\sigma$ with canonical presentation $$\Delta= \langle x_1, x_2, x_3, x_4, x_5: x_1^2=x_2^2=x_3^2=x_4^2=x_5^2=x_1x_2x_3x_4x_5=1 \rangle,$$ and assume the action of $G$ on $S$ to be represented by the epimorphism $\theta: \Delta\to G.$  

The group $G$ cannot be isomorphic to $C_q \rtimes_2 C_4$ because it has a unique involution and therefore every homomorphism $\Delta \to C_q \rtimes_2 C_4$ is not surjective. Similarly, the group $G$ cannot be isomorphic to $C_q \rtimes_4 C_4$ because it has exactly $q$ involutions and all of them are contained in the proper subgroup $\langle a, b^2 \rangle \cong \mathbf{D}_{q}.$ Finally, if $G$ were abelian then $G$ would be isomorphic to a subgroup of $C_2^4;$ but this is not possible. 

All the above ensures that $G$ is isomorphic to the dihedral group of order $4q$ and therefore $S \in \bar{\mathcal{F}}_g^2$. The converse is direct, and the proof of the claim is done.

\section{Proof of Theorem \ref{frontera}} \label{fronterac}

Let $g \geqslant8$ such that $q=g-1$ is prime. Assume $g \equiv 2 \mbox{ mod }4$ and let $S$ be a compact Riemann surface lying in the family $\bar{\mathcal{F}}_g^1.$ We recall that $S$ has a group of automorphisms $G_1$ isomorphic to $$\langle a,b : a^q=b^4=1, bab^{-1}=a^r \rangle=C_q \rtimes_4 C_4,$$where $r$
 is a $4$-th primitive root of the unity in the field of $q$ elements, and that the action is represented by the epimorphism $\Theta: \Delta \to C_{q} \rtimes_4 C_4$ defined by $$\Theta(x_1)=b^2, \,\, \Theta(x_2)=ab^2, \,\, \Theta(x_3)=ab, \,\,  \Theta(x_4)=b^3$$where $\Delta=\langle x_1, x_2, x_3, x_4 : x_1^2=x_2^2=x_3^4=x_4^4=x_1x_2x_3x_4=1 \rangle.$

By classical results due to Singerman \cite{singerman2}, the action of $G_1$ on a (generic) member $S$ of $\bar{\mathcal{F}}_g^1$ can be possibly extended to only an action of a group of order $8q$  with signature $\hat{\sigma}=(0; 2,2,2,4).$ However, as proved in \cite{BJ} for $g \geqslant 18$ and in \cite{C2} for  $g=14,$  there is no compact Riemann surfaces of genus $g$ with a group of automorphisms of order $8q$ acting with signature $\hat{\sigma}$. It follows that:
\begin{enumerate}
\item the interior $\mathcal{F}_g^1$ of the family $\bar{\mathcal{F}}_g^1$ consists of those Riemann surfaces for which $G_1$ agrees with the full automorphism group, and
\item the boundary $$\partial(\bar{\mathcal{F}}_g^1)=\bar{\mathcal{F}}_g^1\setminus \mathcal{F}_g^1=\{ X \in\bar{\mathcal{F}}_g^1 : G_1 \subsetneqq \mbox{Aut}(X) \}$$ consists of finitely many points.
\end{enumerate}

  Note that for each $X \in \partial(\bar{\mathcal{F}}_g^1)$ its full automorphism group has order $\mu q$ where $4$ divides $\mu.$ Then, following \cite{BJ}, there exists a positive integer $\epsilon_1$ such that if $g \geqslant\epsilon_1$ then  either: 
\begin{enumerate}
\item $\mbox{Aut}(X) \cong C_q \rtimes_8 C_8$ acting  with signature $(0;2 ,8,8)$, for  $g \equiv 2 \mbox{ mod }8;$ or 
\item  $\mbox{Aut}(X) \cong (C_q \rtimes_6 C_6) \times C_2$ acting with signature $(0; 2,6,6),$ for $g \equiv 2 \mbox{ mod }3.$ 
\end{enumerate}

The latter case is not possible because $(C_q \rtimes_6 C_6) \times C_2$ does not have elements of order 4, showing that if $g \not\equiv 2 \mbox{ mod } 8$ then $\partial(\bar{\mathcal{F}}_g^1)$ is empty. Let us now assume that $g \equiv 2 \mbox{ mod }8,$ and let $$\Delta'=\langle y_1, y_2, y_3 : y_1^2=y_2^8=y_3^8=y_1y_2y_3=1 \rangle$$be a Fuchsian group of signature $(0; 2,8,8).$ Again, following \cite{BJ}, there are exactly two non-isomorphic  Riemann surfaces $X_1$ and $X_2$ of genus $g \geqslant 18$ such that  $$\mbox{Aut}(X_i) \cong \langle \alpha, \beta: \alpha^{q}=\beta^8=1, \beta \alpha \beta^{-1}=\alpha^{u} \rangle=C_{q} \rtimes_8 C_8,$$where $u$ is a $8$-th primitive root of the unity in the field of $q$ elements, and the action of $\mbox{Aut}(X_i)$ on $X_i$ is determined by the epimorphisms $\Theta_i: \Delta' \to C_{q} \rtimes_8 C_8$  $$\Theta_1(y_1)=\beta^4, \,\, \Theta_1(y_2)=\alpha^{-u}\beta, \,\, \Theta_1(y_3)=\alpha\beta^3$$ $$\Theta_2(y_1)=\beta^4, \,\, \Theta_2(y_2)=\alpha^{u}\beta, \,\, \Theta_2(y_3)=\alpha\beta^7.$$

The subgroup of $\Delta'$ generated by the elements $$\hat{x}_1=y_1, \,\, \hat{x}_2=y_1, \,\, \hat{x}_3=y_3^2, \,\, \hat{x}_4=y_3^6$$is isomorphic to $\Delta,$ and $$\Theta_i(\hat{x}_1)= \beta^4    , \, \Theta_i(\hat{x}_2)= \beta^4   , \,\Theta_i(\hat{x}_3)= \alpha^{1+(-1)^{i+1}u^3}\beta^6  , \,\Theta_i(\hat{x}_4)=\alpha^{u^3+(-1)^iu^2} \beta^7    $$
%
%
%
Thus, for $i=1,2,$ the restriction of $\Theta_i$ to $\Delta \cong \langle \hat{x}_1, \ldots, \hat{x}_4 \rangle$ $$\Delta \to \langle \beta^4, \alpha^{1\pm u^3} \beta^6 \rangle = \langle \alpha, \beta^2 \rangle=C_q \rtimes_4 C_4 $$ defines an action $C_q \rtimes_4 C_4$ on $X_i$ with signature $(0;2,2,4,4),$ showing that  $\partial(\bar{\mathcal{F}}_g^1)$ agrees with $\{X_1, X_2\}$. 

\s

Now, let $S$ be a compact Riemann surface lying in the family $\bar{\mathcal{F}}_g^2.$ We recall that $S$ has a group of automorphisms $G_2$ isomorphic to $$\langle r, s : r^{2q}=s^2=(sr)^2=1\rangle = \mathbf{D}_{2q},$$ and that the action of $G_2$ on $S$ is represented  by the epimorphism $\Theta: \Delta \to \mathbf{D}_{2q}$ $$\Theta(x_1)=s, \,\, \Theta(x_2)=s, \,\, \Theta(x_3)=sr^{q+1}, \,\, \Theta(x_4)= sr, \,\, \Theta(x_5)=r^{q}$$where $\Delta=\langle x_1, x_2, x_3, x_4 : x_1^2=x_2^2=x_3^2=x_4^2=x_5^2=x_1x_2x_3x_4x_5=1 \rangle.$ 

 \s

By \cite{singerman2} the action of $G_2$ on a generic member $S$ of $\bar{\mathcal{F}}_g^2$ cannot be extended. Thus:

\begin{enumerate}
\item  the interior $\mathcal{F}_g^2$ of the family $\bar{\mathcal{F}}_g^2$ consists of those Riemann surfaces for which $G_2$ agrees with the full automorphism group, and
\item the boundary $$\partial(\bar{\mathcal{F}}_g^2)=\bar{\mathcal{F}}_g^2 \setminus \mathcal{F}_g^2=\{ Y \in\bar{\mathcal{F}}_g^2 : G \subsetneqq \mbox{Aut}(Y) \}$$consists of finitely many points and  finitely many one-dimensional families.
\end{enumerate}

By \cite{BJ}, there exists $\epsilon_2$ such that if $g \geqslant\epsilon_2$ and $Y \in \partial(\bar{\mathcal{F}}_g^2)$ then either\begin{enumerate}
\item $\mbox{Aut}(Y) \cong C_q \rtimes_8 C_8$ acting with signature $(0;2 ,8,8)$, for  $g \equiv 2 \mbox{ mod }8,$ or 
\item  $\mbox{Aut}(Y) \cong (C_q \rtimes_6 C_6) \times C_2$ acting with signature $(0; 2,6,6),$ for $g \equiv 2 \mbox{ mod }3.$ 
\end{enumerate}

The former case is not possible because $C_q \rtimes_8 C_8$ does not have elements of order $2q$; thus, if $g \not\equiv 2 \mbox{ mod } 3$ then  $\partial(\bar{\mathcal{F}}_g^2)$ is empty. Let us now assume that $g \equiv 2 \mbox{ mod }3,$ and let $\Delta'$ be a Fuchsian group of signature $(0; 2,6,6)$ with canonical presentation $$\Delta'=\langle y_1, y_2, y_3 : y_1^2=y_2^6=y_3^6=y_1y_2y_3=1 \rangle.$$

Again, following \cite{BJ}, for $g \geqslant  \epsilon_2$ there are exactly two non-isomorphic  Riemann surfaces $Y_1$ and $Y_2$ of genus $g$ with full automorphism group isomorphic to$$ \langle \alpha,\beta,\gamma : \alpha^q=\beta^6=\gamma^2= [\gamma,\alpha]=[\gamma,\beta]=1,  \beta \alpha \beta^{-1}=\alpha^{u} \rangle=(C_q \rtimes_6 C_6) \times C_2,$$where $u$ is a $6$-th primitive root of the unity in the field of $q$ elements, and the action of $\mbox{Aut}(Y_i) $ on $Y_i$ is determined by the epimorphisms $\Theta_i: \Delta' \to (C_q \rtimes_6 C_6) \times C_2$ $$\Theta_1(y_1)=\beta^3, \,\, \Theta_1(y_2)=\alpha^{-u}\beta \gamma, \,\, \Theta_1(y_3)=\alpha\beta^2 \gamma$$ $$\Theta_2(y_1)=\beta^3 \gamma, \,\, \Theta_2(y_2)=\alpha^{-u^2}\beta^2 \gamma, \,\, \Theta_2(y_3)=\alpha\beta.$$

The subgroup of $\Delta'$ generated by $$\tilde{x}_1=y_3^3, \,\, \tilde{x}_2=y_1, \,\, \tilde{x}_3=y_2y_1y_2^{-1}, \,\, \tilde{x}_4=y_2^2y_1y_2^{-2}, \,\, \tilde{x}_5=y_2^3$$is isomorphic to $\Delta,$ and $$\Theta_1(\tilde{x}_1)= \gamma, \, \Theta_1(\tilde{x}_2)= \beta^3, \,\Theta_1(\tilde{x}_3)=\alpha^{-2u}\beta^3, \,\Theta_1(\tilde{x}_4)=\alpha^{2-4u}\beta^3, \, \Theta_1(\tilde{x}_5)=\alpha^{2-2u}\beta^3 \gamma$$ $$\Theta_2(\tilde{x}_1)=\alpha^{2u}\beta^3, \, \Theta_2(\tilde{x}_2)=\beta^3 \gamma, \,\Theta_2(\tilde{x}_3)=\alpha^{-2u^2}\beta^3 \gamma, \, \Theta_2(\tilde{x}_4)=\alpha^2\beta^3\gamma, \, \Theta_2(\tilde{x}_5)=\gamma.$$

Thus, for $i=1,2,$ the restriction of $\Theta_i$ to $\Delta \cong \langle \tilde{x}_1, \ldots, \tilde{x}_4 \rangle$ $$\Delta \to  \langle \alpha \gamma, \beta^3 \rangle=\mathbf{D}_{2q}$$ defines an action $\mathbf{D}_{2q}$ on $Y_i$ with signature $(0;2,2,2,2,2),$ showing that  $\partial(\bar{\mathcal{F}}_g^2)$ agrees with $\{Y_1, Y_2\}$. 

The proof is done.

\section{Proof of Theorem \ref{teo3}} \label{jaco}

Let $g \geqslant8$ such that $q=g-1$ is prime.

We recall the well-known fact that the dihedral group $$\langle r,s : r^{2q}=s^2=(sr)^2=1\rangle=\mathbf{D}_{2q}$$has, up to equivalence, 4 complex irreducible representations of degree one; namely,$$U_1^{\pm}: r \mapsto 1, \, s \mapsto \pm 1, \, \, \, \, \,U_2^{\pm}: r \mapsto -1, \, s \mapsto \pm 1, $$and $q-1$ complex irreducible representations of degree two; namely,
\begin{displaymath}
V_{j} : \, r \mapsto \mbox{diag}(\omega_{2q}^j, \bar{\omega}_{2q}^j) 
 ,\  s \mapsto \left( \begin{smallmatrix}
0 & 1\\
1 & 0 \\
\end{smallmatrix} \right)
\end{displaymath}for $1 \leqslant j  \leqslant q-1$ and $\omega_t=\mbox{exp}(\tfrac{2 \pi i}{t}).$ See, for example, \cite[p. 36]{Serre}.

\s

Following the notations introduced in Subsection \ref{jacos}, the trivial representation $U_1^{+}$ does not belong to $\mathfrak{J}_{\theta},$ where $\theta$ represents the action of $\mathbf{D}_{2q}$ on each member $S$ of the family $\bar{\mathcal{F}}_g^2.$ The following table summarizes the dimension of the vector subspaces of the non-trivial complex irreducible representations of $\mathbf{D}_{2q}$ fixed under the action of the subgroups $\langle r \rangle, \langle s \rangle$ and $\langle sr \rangle$.

\s
\begin{center}
\scalebox{0.7}{
\begin{tabular}{|c|c|c|c|c|c|}  \hline
\, &  $U_1^{-}$ & $U_2^{+}$  & $U_2^{-}$  & $V_j$ \\ \hline
$\langle s \rangle$  & $0$ & $1$ & $0$ & $1$ \\ \hline
$\langle r \rangle$  &  1 & 0 & 0 & $0$ \\ \hline
$\langle sr \rangle$  & 0 & 0 & 1 & $1 $ \\ \hline
\end{tabular}}
\end{center}

\s

It follows that the collection $\{ \langle r \rangle, \langle s \rangle, \langle sr \rangle\}$ is admissible and therefore, by \cite{kanirubiyo}, if $S \in \bar{\mathcal{F}}_g^2$ then there exists an abelian subvariety $P$ of $JS$ such that $$JS \sim JS_{\langle s \rangle} \times  JS_{\langle r \rangle} \times JS_{\langle sr \rangle} \times P.$$

The  covering maps given by the action of $\langle s \rangle, \langle r \rangle$ and $\langle sr \rangle$ ramify over six, two and two values respectively; then, the Riemann-Hurwitz formula implies that $$g_{S_{\langle s \rangle}}=\tfrac{q-1}{2}, \,\,  g_{S_{\langle r \rangle}}=1 \,\, \mbox{ and } \,\, g_{S_{\langle sr \rangle}}=\tfrac{q+1}{2}.$$ 

It follows that $P=0$ and the desired decomposition is obtained.
\begin{rema}
Note that if $S \in \bar{\mathcal{F}}_g^2$ then $JS$ contains an elliptic curve.
\end{rema}

We now assume $g \equiv 2 \mbox{ mod }4.$ Let $r$ be a $4$-th primitive root of the unity in the field of $q$ elements, write $m=\tfrac{q-1}{4}$ and choose $1 \leqslant k_1, \ldots, k_m \leqslant q-1$ such that \begin{equation*}  \{1, \ldots, q-1\} = \sqcup_{j=1}^m \{\pm k_j, \pm rk_j \},\end{equation*} where the symbol $\sqcup$ stands for disjoint union. Then $$\langle a,b : a^q = b^4 = 1, bab^{-1}=a^r \rangle =C_q \rtimes_4 C_4$$has, up to equivalence, $m$ complex irreducible representations of degree 4, given by
\begin{displaymath}
V_{j} : a \mapsto \mbox{diag}(\omega_{q}^{k_j}, \omega_{q}^{k_jr}, \omega_{q}^{-k_j}, \omega_{q}^{-k_jr}),  \,\,\,\,\, b \mapsto \left( \begin{smallmatrix}
0 & 1 & 0 & 0 \\
0 & 0 & 1 & 0 \\
0 & 0 & 0 & 1 \\
1 & 0 & 0 & 0 \\
\end{smallmatrix} \right)  \mbox{ where } \omega_t=\mbox{exp}(\tfrac{2 \pi i}{t}),
\end{displaymath}and four complex irreducible representations of degree 1, given by $$U_l : a \mapsto 1, \, b \mapsto \omega_4^l, \, \mbox{ for } \, 0 \leqslant l \leqslant 3$$(see, for example, \cite[p. 62]{Serre}). Choose four pairwise different integers $t_1, t_2, t_3, t_4 \in \{1, \ldots, q-1\},$ and consider the following subgroups of $C_q \rtimes_4 C_4$ $$\langle a \rangle = C_q  \mbox{ and } \langle a^{t_i}b \rangle = C_4 \mbox{ for } 1 \leqslant i \leqslant 4.$$

The trivial representation $U_0$ does not belong to $\mathfrak{J}_{\theta},$ where $\theta$ represents the action of $C_q \rtimes_4 C_4$ on each member $S$ of the family $\bar{\mathcal{F}}_g^1.$ The  dimension of the vector subspaces of the non-trivial complex irreducible representations of $C_q \rtimes_4 C_4$ fixed under the action of the subgroups $\langle a \rangle$ and $\langle a^{t_i}b \rangle$ is: $$ U_l^{\langle a \rangle}=V_j^{\langle a^{t_i}b \rangle }=1 \,\,\, \mbox{ and } \,\,\, U_l^{\langle a^{t_i}b \rangle}=V_j^{\langle a \rangle}=0.$$

Thus $\{ \langle a \rangle, \langle a^{t_1}b \rangle, \ldots, \langle a^{t_4}b \rangle\}$ is admissible and therefore, by \cite{kanirubiyo},  if $S \in \bar{\mathcal{F}}_q^1$ then $$JS \sim JS_{\langle a \rangle} \times \Pi_{i=1}^4JS_{\langle a^{t_i}b \rangle} \times Q \cong JS_{\langle a \rangle} \times (JS_{\langle b \rangle})^4 \times Q ,$$for some abelian subvariety $Q$ of $JS$, where the isomorphism follows  after noticing that each $\langle a^{t_i}b\rangle$ and $\langle b \rangle$ are conjugate. The covering map $S \to S_{\langle a \rangle}$ is unbranched, and the covering map $S \to S_{\langle b \rangle}$ ramifies over four values, two marked with 2 and two marked with 4. Then, the Riemann-Hurwitz formula implies that $$g_{S_{\langle a \rangle}}=2 \, \mbox{ and } \, g_{S_{\langle b \rangle}}=m.$$

Thereby, $Q=0$ and the decomposition of $JS$ stated in Theorem \ref{teo3} is done.

\section{The case $g-1$ not prime} \label{noprimo}

\subsection*{Example 1.} Let $n \geqslant3$ be an integer, and consider the group $$\langle x,y,z : x^4=z^n=1, x^2=y^2, yxy^{-1}=x^3, [x,z]=[y,z]=1\rangle = \mathbf{Q}_8 \times C_n$$where $\mathbf{Q}_8$ denotes the quaternion group, and let $\Delta$ be a Fuchsian group of signature $(1;2)$ with canonical presentation $\Delta= \langle a_1, b_1, x_1 : [a_1, b_1]x_1=x_1^2=1 \rangle.$ For each $n$ odd, the epimorphism $\Delta \to \mathbf{Q}_8 \times C_n$ given by $a_1 \mapsto x, \, b_1 \mapsto yz, \, x_1 \mapsto y^2$ guarantees the existence of a complex one-dimensional family of compact Riemann surfaces of genus $g=1+2n$ with a group of automorphisms of order $8n=4g-4$ isomorphic to $\mathbf{Q}_8 \times C_n$ acting with signature $(1;2).$

\subsection*{Example 2.}  \label{angel} Let $n \geqslant3.$ Choose $m \in \{\pm 1, 2^{n-1}\pm 1\},$ consider the group  $$\langle r,s,t : r^{2^n}=s^2=(sr)^2=t^2=1, trt=r^m, tst=s \rangle = \mathbf{D}_{2^n} \rtimes C_2,$$and let $\Delta$ be a Fuchsian group of the  signature $\sigma=(0;2,2,2,2,2)$ with  presentation $$\Delta=\langle x_1, x_2, x_3, x_4,  x_5 : x_1^2=x_2^2=x_3^2=x_4^2=x_5^2=x_1x_2x_3x_4x_5=1 \rangle.$$

The epimorphism $\Delta \to \mathbf{D}_{2^n} \rtimes C_2$ given by $$x_1 \mapsto sr, \, x_2 \mapsto sr, \, x_3 \mapsto s, \, x_4 \mapsto t, x_5 \mapsto st$$guarantees the existence of a complex two-dimensional family of compact Riemann surfaces of genus $g=2^n+1,$ with a group of automorphisms of order $2^{n+2}=4g-4$ isomorphic to $\mathbf{D}_{2^n} \rtimes C_2$ acting with signature $\sigma.$ Two different choices of $m$ yield non-isomorphic groups, showing that if $g-1=2^n$ then there exist at least four pairwise non-isomorphic groups of order $4g-4$ acting on compact Riemann surfaces of genus $g$ with the same signature $\sigma.$

\section*{Acknowledgments}
The author is grateful to his colleague Angel Carocca who generously told him how to construct Example 2 in Section \ref{noprimo}.

\end{document}